\documentclass{amsart}

\usepackage[utf8]{inputenc}
\usepackage{amsmath,amssymb,bbm,enumitem,array,multicol}
\usepackage[pdfencoding=auto]{hyperref}
\usepackage{amsthm}
\usepackage{listings}[2013/08/05]
\usepackage[capitalize]{cleveref}
\usepackage[square,numbers,sort]{natbib}

\hypersetup{pdfstartview=}
\newcolumntype{C}{>{$}c<{$}}

\newcommand\BC{\mathbb C}

\newcommand\BF{\mathbb F}
\newcommand\BZ{\mathbb Z}
\newcommand\BN{\mathbb N}
\newcommand\BQ{\mathbb Q}
\newcommand\BR{\mathbb R}

\newcommand{\cyc}[1]{\langle #1 \rangle}

\def\BHM#1.#2.#3.#4.{{^{#1}_{#3}\mathcal B^{#2}_{#4}}}

\newcommand\comm\curlyvee
\newcommand\cocomm\curlywedge

\newcommand{\Sp}{\text{Sp}}
\newcommand{\PSp}{\text{PSp}}
\newcommand{\ip}[2]{\langle#1,#2\rangle}
\newcommand{\legendre}[2]{\genfrac{(}{)}{}{}{#1}{#2}}

\DeclareMathOperator{\tr}{tr}

\theoremstyle{plain}
\newtheorem{thm}{Theorem}[section]

\newtheorem{cor}[thm]{Corollary}
\newtheorem{prop}[thm]{Proposition}
\newtheorem{lem}[thm]{Lemma}

\theoremstyle{definition}
\newtheorem{df}[thm]{Definition}
\newtheorem{example}[thm]{Example}
\newtheorem{question}[thm]{Question}

\theoremstyle{remark}
\newtheorem{rem}[thm]{Remark}

\crefname{lem}{Lemma}{Lemmas}
\crefname{thm}{Theorem}{Theorems}
\crefname{cor}{Corollary}{Corollaries}
\crefname{prop}{Proposition}{Propositions}
\crefname{example}{example}{examples}
\crefname{df}{Definition}{Definitions}
\crefname{equation}{equation}{equations}

\numberwithin{equation}{thm}

\def\clap#1{\hbox to 0pt{\hss#1\hss}}

\newcommand\inv{^{-1}}
\newcommand{\kk}{\mathbbm{k}}

\DeclareMathOperator{\Rep}{Rep}

\newcommand{\GL}[2]{\text{GL}_{#1}({#2})}

\newlist{lemenum}{enumerate}{1}
\setlist[lemenum]{label=\roman*), ref=\textup{\thethm~(\roman*)}}
\crefalias{lemenumi}{lemma}

\title[FSZ properties of symplectic groups]{A family of non-FSZ finite symplectic groups}
\author{Marc Keilberg}
\email{keilberg@usc.edu}
\begin{document}
\begin{abstract}
Let $p$ be an odd prime with $p\equiv1\bmod 4$.  Then for any odd power $q$ of $p$ and a positive integer $j$ we show that the groups $\Sp_{p^j+1}(q),\PSp_{p^j+1}(q)$, and their Sylow $p$-subgroups are
non-$FSZ_{p^j}$.
\end{abstract}
\maketitle
\section{Introduction}
Fusion categories are a topic of significant interest in mathematical physics, and one of the major invariants they are studied with are the higher Frobenius-Schur indicators \citep{KSZ2,NS07b}.  These generalize the notion of Frobenius-Schur indicators for representations of finite groups, are algebraic integers in a cyclotomic field, and are akin to structure constants for certain algebraic objects within the category \citep{BJT:Rotations}.  In the classical case of finite groups the indicators are in fact always integers, but this need not be true for more arbitrary categories. The question then arises as to which fusion categories share this property of having all integer indicators. A particular subcase is to ask this question for the categorical centers of $\Rep(G)$, with $G$ a finite group. Groups for which this integrality property holds for this categorical center have been called $FSZ$ groups.

\Citet{IMM} were the first to show that non-$FSZ$ groups exist, with the assistance of \citet{GAP4.8.4}.  Several other choices and families of $G$ were checked either by direct computation of the indicators \citep{Co,K,K2}, or via the more combinatorial methods of Iovanov, et al. \citep{K16:p-examples}. A more character theoretic test is provided by \citep{PS16}, which in particular provides a number of substantially more efficient GAP routines.  \Citet{PS16} used these methods to check small order simple groups, and the author subsequently adapted them to check all sporadic simple groups, small order perfect groups, and $\PSp_6(5)$ \citep{K16:Sporadics}.  Absent to date, however, has been an abstract proof that a group is non-$FSZ$ using this character theoretic approach.  In this paper we fill this gap by establishing that certain (projective) symplectic groups and their Sylow subgroups in defining characteristic are non-$FSZ$ via these methods.

The paper begins in \cref{sec:prelims} by recalling and establishing several key properties of finite fields and $FSZ$ groups.  In \cref{sec:choices} we help motivate the choices we make in the remainder of the paper by discussing ways to efficiently establish that a group is non-$FSZ$.  We then begin our investigations into the (projective) symplectic groups in \cref{sec:sylows} by investigating certain properties of their Sylow subgroups in defining characteristic.  We are then able to establish that these Sylow subgroups are non-$FSZ$ under certain technical restrictions on the field and dimension.  We do this via characters in \cref{sec:sylow-characters}, and also provide a combinatorial proof in \cref{sec:sylow-counting}.  We conclude by establishing our main result in \cref{sec:symplectic}, which is that under the same technical restrictions the (projective) symplectic groups are also non-$FSZ$. This proof is done strictly with the character theory ideas, as the author was unable to find a combinatorial proof.  This result provides the first proof that there are infinitely many non-$FSZ$ finite simple groups.  We make various remarks and observations throughout concerning the aforementioned technical restrictions.  In particular, the restriction on the dimension seems largely artificial, while those on the field appear to be essential.  The author was unable to decide these matters definitively, however.

\section{Preliminaries}\label{sec:prelims}
We need to recall several fundamental properties of finite fields and $FSZ$ groups.  Our reference for the properties of finite fields is \citep{IR:NumberTheory}. We let $\BN$ denote the set of positive integers, and all groups will be finite unless otherwise noted.
\subsection{Quadratic Residues}
\begin{df}
  Given a field $\BF$, we say that $x\in\BF$ is a quadratic residue if there exists $y\in\BF$ such that $y^2=x$.

  We let $QR(\BF)$, or $QR(q)$ when $\BF\cong \BF_q$ is a finite field, or just $QR$ when the field is understood from the context, denote the set of all quadratic residues of $\BF$.
\end{df}
We identify the prime subfield of a finite field $\BF_{p^n}$ with the set of integers $\{0,1,...,p-1\}$ in the usual fashion.
\begin{example}
  \begin{align*}
    QR(5)&=\{0,1,4\},\\
    QR(7)&=\{0,1,2,4\},\\
    QR(11)&=\{0,1,3,4,5,9\}.
  \end{align*}
\end{example}
\begin{example}
  Since a finite field is completely determined by its order, any finite field extension $K\supset \BF_q$ with $[K:\BF_q]$ even has $\BF_q\subset QR(K)$.  Subsequently, if $q$ is an even power of $p$ then $QR(q)$ contains the prime subfield $\{0,1,...,p-1\}$ of $\BF_q$.
\end{example}

Our results in \cref{sec:sylow-characters,sec:sylow-counting} will rely on counting how many quadratic residues lie in certain sets.  The last part of the next result answers the question of how many times we can write a given element as a difference of quadratic residues.
\begin{lem}\label{lem:residues}
The following hold when $q$ is a power of an odd prime $p$.
  \begin{lemenum}
    \item  $|QR(q)| = \displaystyle{\frac{q+1}{2}}$.\label{lempart:residues-1}
    \item If $x,y\in QR(q)$, then $xy\in QR(q)$.  If $0\neq x\in QR(q)$ and $y\not\in QR(q)$, then $xy\not\in QR(q)$.
    \item $-1\in QR(q)$ if and only if $p\equiv 1\bmod 4$ or $q$ is an even power of $p$.
    \item  For any $0\neq c\in\BF_q$
    \[ |QR\cap \{x+c\ : \ x\in QR\}| = \begin{cases}
      \frac{q+3}{4},& -1\in QR\text{ and }c\in QR\\
      \frac{q-1}{4},& -1\in QR\text{ and }c\not\in QR\\
      \frac{q+1}{4},& \text{otherwise}
    \end{cases}.\]\label{lempart:residues-3}
  \end{lemenum}
\end{lem}
\begin{proof}
  For the first claim, the map $a\mapsto a^2$ defines a group endomorphism of $\BF_q^\times$, the multiplicative group of the field.  This group has order $q-1$, and the endomorphism has kernel $\{\pm 1\}$, which gives the first claim.  The second claim then also follows, as $QR(q)\setminus\{0\}$ is an index 2 subgroup of $\BF_q^\times$.

  For the third claim, it is well-known that $-1\in QR(p)$ if and only if $p\equiv 1\bmod 4$; indeed, when $p\equiv 1\bmod 4$ then by Wilson's lemma $(p-1)! = ((p-1)/2)!^2\equiv-1\bmod p$.  So if $p\equiv 3\bmod 4$, then $x^2+1$ is irreducible over $\BF_p$, and so $-1\in QR(q)$ implies that $\BF_q$ is a splitting field for $x^2+1$. Therefore $\BF_q$ is an extension of $\BF_{p^2}$, from which it follows that $q$ is an even power of $p$, as desired.
  
  We now consider the fourth claim. The cardinality desired is given by counting the number of distinct ordered pairs $(x^2,y^2)$ where $(x,y)\in\BF_q\times\BF_q$ satisfy $x^2=y^2+c$.  We first count the number of such pairs $(x,y)$, from which we can then obtain the number of distinct pairs $(x^2,y^2)$.
  
  We have that $x^2=y^2+c$ if and only if $x^2-y^2=(x+y)(x-y)=c$.  Since $c\neq 0$, $x+y$ and $x-y$ are both non-zero.  Moreover, if $x+y=a\in\BF_q^\times$, then $x-y=ca\inv$, and we can then solve for $x$ and $y$.  Therefore the set of all pairs $(x,y)$ satisfying $x^2=y^2+c$ has cardinality $q-1$.  Given any such solution $(x,y)$, then $(\pm x,\pm y)$ are also solutions, and are exactly the solutions yielding the pair $(x^2,y^2)$. If both $x,y$ are non-zero, then since $q$ is odd this identifies exactly four $(x,y)$ pairs.  We then readily observe that $y=0$ occurs in a solution if and only if $c\in QR$, and there are two values of $x$ such that $x^2=c$.  Similarly, $x=0$ occurs in a solution if and only if $-c\in QR$, and there are two values of $y$ such that $y^2=-c$.
  
  Therefore, if $\pm c\in QR$ then there are $(q-1)-4=q-5$ solutions $(x,y)$ with both values non-zero, which yields $(q-5)/4$ distinct pairs of values $(x^2,y^2)$. The solutions with $x=0$ or $y=0$ add the pairs $(c,0),(0,-c)$.  Combined, the desired cardinality when $\pm c\in Q$ is precisely $2+(q-5)/4 = (q+3)/4$.  By the second part of the result, $\pm c\in QR$ if and only if $-1,c\in QR$. Therefore we obtain one of the three stated cases.
  
  When neither $c$ nor $-c$ are in $QR$, then there are no solutions with $x=0$ or $y=0$, and so we obtain $(q-1)/4$ distinct pairs $(x^2,y^2)$. By the second part of the result, we have that $c,-c\not\in QR$ if and only if $-1\in QR$ and $c\not\in QR$, yielding the second of the three stated cases.
  
  In all other cases we have $(q-1)-2=q-3$ solutions $(x,y)$ with both values non-zero; and exactly two solutions with one value zero, both yielding the same pair of squares.  Therefore, the desired cardinality is $(q-3)/4 + 1 = (q+1)/4$, as claimed.
  
  This completes the proof.
\end{proof}

\subsection{Additive characters of fields}\label{sub:field-chars}
Given a field $\BF$, we let $\BF^+$ denote its underlying additive group.  The results of \cref{sec:sylow-characters} are dependent upon the distribution of quadratic residues in pre-images of group homomorphisms $\BF_{p^n}^+\to\BF_p^+$.  In order to state our main result on this, we need to review a number of definitions.

There is a canonical group homomorphism $\tr\colon\BF_{p^n}^+\to\BF_p^+$, called the trace, defined by $\tr(x) = \sum_{i=0}^{n-1}x^{p^i}$.  More generally, traces are defined for any finite Galois extension via summing Galois conjugates.  There is a bijection between $\BF_{p^n}$ and group homomorphisms $\BF_{p^n}^+\to\BF_p^+$ given by $z\mapsto (x\mapsto \tr(zx))$.  The trivial homomorphism corresponds to $z=0$, and all other homomorphisms are non-trivial (and so surjective).  We also have a canonical character $e_{p^n}\colon\BF_{p^n}^+\to\BC^\times$ given by $x\mapsto \exp(2\pi i \tr(x)/p)$.

The Legendre symbol $\legendre{x}{q}$ for a field of order $q$ and $x\in\BF_q$ is defined by
\begin{equation}
\legendre{x}{q} =
  \begin{cases}
    0,& x=0\\
    1,& 0\neq x\in QR(q)\\
    -1,&\mbox{otherwise}
  \end{cases}.
\end{equation}
The Legendre symbol is multiplicative, and in particular restricts to a character $\BF_q^\times\to\BC^\times$.

We also define
\begin{equation}
  G(q) = \sum_{x\in\BF_q} \legendre{x}{q}e_q(x).
\end{equation}
These are known as Gauss sums, and satisfy
\begin{equation}
  G(p) = \begin{cases}
    \sqrt{p},& p\equiv 1\bmod 4\\
    i\sqrt{p},& p\equiv 3\bmod 4
  \end{cases},
\end{equation}
and $G(p^n) = -(-G(p))^n$.

\begin{thm}\label{thm:char-residues}
  Let $p$ be an odd prime, $n\in\BN$, and $z\in\BF_{p^n}^\times$.  Let $\tr_z\colon\BF_{p^n}^+\to\BF_p^+$ be the group homomorphism given by $\tr_z(x)=\tr(zx)$.
  \begin{enumerate}
    \item If $n$ is odd, then
    \[ |\ker(\tr_z)\cap QR(p^n)| = \frac{p^{n-1}+1}{2}.\]
    \item If $n$ is even, then
    \[ |\ker(\tr_z)\cap QR(p^n)| = \frac{p^n + (p-1)\legendre{z}{p^n}G(p^n)+p}{2p}.\]
    \item If $n$ is odd and $y\in \BF_p^\times$ then
    \[ |\tr_z\inv(y)\cap QR(p^n)| = \frac{p^n+\legendre{z}{p^n}\legendre{y}{p}G(p)G(p^n)}{2p}.\]
    \item If $n$ is even and $y\in\BF_p^\times$ then
    \[ |\tr_z\inv(y)\cap QR(p^n)| = \frac{p^n- \legendre{z}{p^n}G(p^n)}{2p}.\]
  \end{enumerate}
\end{thm}
\begin{proof}
  A proof for the first two parts is given in \citep{Jyrki}.  We will adapt this proof to cover the remaining cases.

  Let notation and assumptions be as in the statement, and define $q=p^n$.

  Fix any $y\in\BF_p$. Since $\tr_z$ is surjective the pre-image $\tr_z\inv(y)$ is non-empty. Observe that $\tr_z\inv(0)=\ker(\tr_z)$. Let $\tilde{y}\in\BF_q$ be any value such that $\tr_z(\tilde{y})=y$.

  We consider the sum
  \begin{equation}\label{eq:qr-char-sum}
    \sum_{a=0}^{p-1} \sum_{x\in\BF_{q}} (1+\legendre{x}{q})e_q(az(x-\tilde{y})).
  \end{equation}
  The terms with $x=0$ in the above contribute the value
  \[ \sum_{a=0}^{p-1} e_q(-az\tilde{y}) = p \delta_{y,0}\]
  by orthogonality of characters.  For the terms with $x\neq 0$, using the definition of the Legendre symbol and orthogonality of characters we compute that
  \begin{align*}
    \sum_{a=0}^{p-1} \sum_{x\in\BF_{q}^\times} (1+\legendre{x}{q})e_q(az(x-\tilde{y})) &= 2\sum_{0\neq x\in QR(q)}\sum_{a=0}^{p-1}e_q(az(x-\tilde{y}))\\
    &= 2 \sum_{0\neq x\in QR(q)} \begin{cases}
      p,& \tr_z(x)=y\\
      0,& \mbox{ otherwise}
    \end{cases}\\
    &= 2p|\tr_z\inv(y)\cap (QR(q)\setminus\{0\})|.
  \end{align*}
  As $0$ is in the preimage $\tr_z\inv(y)$ if and only if $y=0$, we have
  \[ \sum_{a=0}^{p-1} \sum_{x\in\BF_{q}} (1+\legendre{x}{q})e_q(az(x-\tilde{y})) = 2p|\tr_z\inv(y)\cap QR(q)| - p\delta_{y,0}.\]
  We may now proceed to compute the cardinalities in question by evaluating \cref{eq:qr-char-sum} in a different way.

  The terms with $a=0$ in \cref{eq:qr-char-sum} contribute the value
  \begin{equation}
    \sum_{x\in \BF_q}(1+\legendre{x}{q}) = p^n
  \end{equation}
  by definition of the Legendre symbol and \cref{lempart:residues-1}.

  On the other hand, for $1\leq a\leq p-1$ we have
  \begin{align*}
    \sum_{x\in \BF_q}e_q(az(x-\tilde{y}))(1+\legendre{x}{q}) &= \sum_{x\in\BF_q}e_q(az(x-\tilde{y})) + \sum_{x\in\BF_q} \legendre{x}{q} e_q(az(x-\tilde{y}))\\
    &=\sum_{x\in\BF_q} \legendre{x}{q} e_q(az(x-\tilde{y})
  \end{align*}
  by orthogonality of characters. So by a variable substitution, the multiplicativity of the Legendre symbol, and the identities for Gauss sums we conclude that
  \begin{align*}
    \sum_{x\in \BF_q}e_q(az(x-\tilde{y}))(1+\legendre{x}{q}) &= e_q(-az\tilde{y}) \sum_{x\in\BF_q} \legendre{x}{q}e_q(azx)\\
    &= e_q(-az\tilde{y}) \legendre{a}{q}\legendre{z}{q} \sum_{x\in\BF_q} \legendre{x}{q}e_q(x)\\
    &= \legendre{a}{q}e_q(-az\tilde{y})\legendre{z}{q} G(p^n).
  \end{align*}

  Combining results so far, it follows that
  \begin{equation}\label{eq:qr-card-all}
    2p | \tr_z\inv(y)\cap QR(q)| = p^n + \delta_{y,0}p + \Big(\sum_{a=1}^{p-1}\legendre{a}{q}e_q(-az\tilde{y})\Big)G(p^n)\legendre{z}{q}.
  \end{equation}

  To obtain the desired formulas, we need to evaluate the remaining sum.

  When $n$ is even we have that $\legendre{a}{q}=1$ for all $1\leq a\leq p-1$, so that
  \begin{equation}\label{eq:qr-card-even}
    2p | \tr_z\inv(y)\cap QR(q)| = p^n + (\delta_{y,0}p-1)G(p^n)\legendre{z}{q} + \delta_{y,0}p.
  \end{equation}
  This then gives the desired formulas in the case $n$ is even.

  On the other hand, when $n$ is odd we have $\legendre{a}{q}=\legendre{a}{p}$ for all $1\leq a\leq p-1$, so since $\tr$ is $\BF_p$-linear we have
  \begin{equation}
    \sum_{a=1}^{p-1}\legendre{a}{q}e_q(-az\tilde{y}) = \sum_{a=1}^{p-1}\legendre{a}{p}e_p(-ay).
  \end{equation}
  Another classical result of Gauss shows that this last value is precisely $\legendre{y}{p}G(p)$.  Thus when $n$ is odd we have
  \begin{equation}\label{eq:qr-card-odd}
    2p | \tr_z\inv(y)\cap QR(q)| = p^n +\legendre{y}{p}\legendre{z}{q}G(p)G(p^n) + \delta_{y,0}p.
  \end{equation}
  This gives the desired formulas when $n$ is odd, and so completes the proof.
\end{proof}

\subsection{FSZ Groups}\label{sub:FSZ}
The notion of $FSZ$ groups was introduced by \citet{IMM}. These originate from deciding when certain invariants \citep{KSZ2} of a representation category associated to a given group \citep{DPR} are in fact all integers (the 'Z' in $FSZ$). These invariants constitute generalizations of the classical Frobenius-Schur indicators of a finite group (giving the 'FS' in $FSZ$).  While these invariants have been computed explicitly by hand or with a computer for several (families of) groups \citep{K,K2,PS16,Co}, there are ways of deciding when these invariants are integers or not without explicit calculation of the invariants.  In particular, \citet{IMM} showed that we can use the following definition.

\begin{df}\label{df:FSZ}
  Let $G$ be a finite group.  For any $m\in\BN$ and $u,g\in G$ define $G_m(u,g) = \{ a\in G\ | \ a^m = (au)^m = g\}$.

  We say that $G$ is $FSZ_m$ if for all $n\in\BN$ coprime to $|G|$ we have $|G_m(u,g)|=|G_m(u,g^n)|$ for all $u,g\in G$.

  We say that $G$ is $FSZ$ if it is $FSZ_m$ for all $m\in\BN$.
\end{df}
Indeed, $G$ is $FSZ$ if and only if it is $FSZ_m$ for every $m$ dividing $\exp(G)$, the exponent of $G$.  Moreover, $G_m(u,g)=\emptyset$ if $u\not\in C_G(g)$, and in all cases $G_m(u,g)\subseteq C_G(g)$.

We are often concerned with deciding for which $u$ or $g$ the equality of sets from \cref{df:FSZ} can fail (often when given the other one).  It helps to ease the discussion of such matters if we introduce the following.
\begin{df}\label{df:FSZ2}
Let $m\in\BN$.

We say $G$ is non-$FSZ_m$ \textit{at} $g\in G$ if for some $u\in G$ and some $n\in\BN$ with $(n,|G|)=1$ we have $|G_m(u,g)|\neq|G_m(u,g^n)|$.

We say $G$ is non-$FSZ_m$ \textit{over} $u\in G$ if for some $g\in G$ and some $n\in\BN$ with $(n,|G|)=1$ we have $|G_m(u,g)|\neq |G_m(u,g^n)|$.
\end{df}
\begin{rem}
  It is a matter of preference or convenience if one chooses to emphasize the $g$ or the $u$ term in $|G_m(u,g)|$.  If one is primarily interested in the modules yielding non-integer indicators, then it is the $u$ value that we are most interested in.  However most tests of the $FSZ$ properties, both in theory and practice, focus on computations in $C_G(g)$, and so the value $g$ is usually of primary interest.
\end{rem}

\begin{example}\label{ex:IMM}
\citet{IMM} established that several types of groups are $FSZ$ using the preceding definition, of which we list a few.
  \begin{itemize}
    \item Symmetric and alternating groups (see also \citep{Etingof:SymFSZ}).
    \item All regular $p$-groups.
    \item The irregular $p$-group $\BZ_p\wr_r\BZ_p$.
    \item $\text{PSL}_2(q)$ for any prime power $q$.
    \item The Matthieu groups.
    \item Any group with square-free exponent.
  \end{itemize}
They also showed that there are non-$FSZ_5$ groups of order $5^6$ using \citet{GAP4.8.4}, thus proving that non-$FSZ$ groups exist.
\end{example}
\begin{example}
  The author has constructed several other families of $FSZ$ and non-$FSZ$ $p$-groups for $p>3$ in \citep{K16:p-examples,K17:wreath}.  Whether or not non-$FSZ$ 2-groups or 3-groups exist remains an open question.
\end{example}

These examples all relied on computing cardinalities or otherwise exhibiting bijections between the sets from \cref{df:FSZ}. \citet{PS16} has developed a more character-theoretic criterion for the $FSZ$ properties, which we now detail.

\begin{df}\label{df:psi-beta}
  Let $K$ be a group, $z\in K$, and $\chi$ an irreducible (complex) character of $C_K(z)$.

  For $m\in \BZ$ we define the class function $\psi_{m,z}\colon C_K(z)\to\{0,1\}$ by $\psi_{m,z}(x)=1$ if $x^m=z$, and $\psi_{m,z}(x)=0$ otherwise.

  We also define
  \begin{align}
    \beta_m(\chi,z) &=\| \sum_{\substack{a\in C_K(z)\\a^m=z}}\chi(a)\|^2\\
    &= |C_K(z)|^2 \|\ip{\chi}{\psi_m^z}\|^2
  \end{align}
\end{df}

The following then gives the criterion we wish to use.
\begin{thm}[Schauenburg's Criterion]\citep[Theorem 8.4]{PS16}\label{thm:peter}
Let $K$ be a group and $m\in\BN$.  Then $K$ is $FSZ_m$ if and only if for all $z\in K$ and irreducible characters $\chi$ of $C_K(z)$ we have $\beta_m(\chi,z)\in\BQ$.
\end{thm}
\begin{rem}
  \citep[Theorem 8.4]{PS16} is more general than what we have stated here, and our definition of $\beta$ differs by a rational multiple from the one used in \citep[Lemma 3.4]{PS16}.  We have stated the result and definition as above to more easily serve our more limited needs.
\end{rem}

\citet{PS16} used this to construct a very useful \citet{GAP4.8.4} algorithm to test if a group is $FSZ$ or not.  To date, the preceding theorem has not been used to abstractly prove a group is non-$FSZ$.  Indeed, all current proofs that a group is $FSZ$ or non-$FSZ$ that do not rely on a computer calculation are ultimately all based in investigating \cref{df:FSZ}. This paper will fill this gap and provide several infinite families of non-$FSZ$ groups that utilize Schauenburg's criterion.  Indeed, while the Sylow subgroups we will consider admit a combinatorial proof using \cref{df:FSZ}, we know of no proof that doesn't utilize Schauenburg's criterion for the groups $\Sp_{2n}(q)$ and $\PSp_{2n}(q)$ that are considered herein.  This also provides the first known infinite family of non-$FSZ$ simple groups.

\begin{example}
  The aforementioned GAP algorithm was used \citep{PS16} to show that the simple groups $G_2(5)$ and $HN$ are non-$FSZ_5$.  Shortly afterwards, the author \citep{K16:Sporadics} showed the only other non-$FSZ$ sporadic simple groups were $Ly$, $B$, and $M$.  This also relied on GAP calculations, and made mixed use of Schauenburg's criterion and \cref{df:FSZ}.
\end{example}

While Schauenburg's criterion seems the natural thing to prefer at an abstract level, we note that there is information that a proof a group is non-$FSZ$ utilizing \cref{df:FSZ} provides which a proof using \cref{thm:peter} does not, in the following sense.  In the sets $G_m(u,g)$ the module that would yield the non-integer indicator is associated in a natural way to the element $u$, and not the element $g$. We thus obtain at least partial information about a particular module in the category in question.  In Schauenburg's criterion a value $z$ such that $\beta_m(\chi,z)\not\in\BQ$ tells us that for some $u\in G$ and $n\in\BN$ coprime to $|G|$ we have $|G_m(u,z)|\neq |G_m(u,z^n)|$, but the criterion provides no clear way of deciding which choices of $u$ or $n$ will demonstrate this, and so provides no information about any specific modules.

Since we are only interested in the rationality of $\beta_m(\chi,z)$ here, and not the exact values, we adopt the following notation.
\begin{df}
  We define a relation $\sim_\BQ$ on $\BR$ by $x\sim_\BQ y$ if and only there exists $s,t\in\BQ$ with $t\neq 0$ and $x=s+t y$.  This is an equivalence relation on $\BR$.
\end{df}
\begin{example}
  The rational numbers form one equivalence class under $\sim_\BQ$.  Any set of elements in $\BR$ which is linearly independent over $\BQ$ has the property that all of its elements are in distinct equivalence classes under $\sim_\BQ$.  So $1,\sqrt{2},$ and $\sqrt{3}$ all represent distinct equivalence classes under $\sim_\BQ$.
\end{example}
We have chosen the above definition to clearly reflect what our procedure will be later: to show that $\beta_m(\chi,z)$ is irrational by expressing $\beta_m(\chi,z) = s_1+t_1\gamma_1$, $\gamma_1=s_2+t_2\gamma_2$, etc., without having to explicitly the compute the rational values $s_i,t_i$, until finally reaching a value $\gamma_i$ whose rationality we can decide.  It will actually not be terribly difficult to explicitly determine all of the suppressed rational constants, but this is one more piece of bookkeeping we opt to not be distracted by here.
\section{Choosing good characters}\label{sec:choices}
The principle difficulty in applying either \cref{df:FSZ} or \cref{thm:peter} to show a group is non-$FSZ$ is deciding which possibilities to test first. While brute force methods can yield a number of successes \citep{PS16,IMM}, they are only suitable to a rather small selection of groups.  The author proposed several simple guiding principles for applying these tests in a more efficient manner in \citep{K16:Sporadics}.  We will propose several similar things in this section, with an eye towards motivating the choices we will make in the subsequent sections.

Our first observation is used to single out linear characters as good first candidates for applying Schauenburg's criterion.
\begin{lem}\label{lem:lin-expand}
  If $G$ is group with $z\in Z(G)$, then for any $\chi\in\widehat{G}$ and $m\in\BN$ we have
  \begin{align}\label{eq:beta-lin}
    \beta_m(\chi,z) = \sum_{u\in G} |G_m(u,z)| \chi(u).
  \end{align}
\end{lem}
\begin{proof}
  Let assumptions be as in the statement.  Expanding the definition of $\beta_m(\chi,z)$ and using multiplicativity of $\chi$ we have
  \begin{align*}
    \beta_m(\chi,z) &= \sum_{a^m=b^m=z}\chi(a)\overline{\chi(b)}\\
    &= \sum_{c\in\BC} | \{(a,b)\in G^2 \ | \ a^m=b^m=z,\ \chi(a)\chi(b\inv)=c\}|c\\
    &=\sum_{c\in\BC} | \{(a,b)\in G^2 \ | \ a^m=b^m=z,\ \chi(ab\inv)=c\}|c\\
    &=\sum_{c\in\BC} \Big(\sum_{\substack{u\in G\\\chi(u)=c}} | \{b\in G \ | b^m=(ub)^m=z|\Big)c\\
    &= \sum_{u\in G} |G_m(u,z)|\chi(u),
  \end{align*}
  as desired.
\end{proof}
Therefore the set cardinalities appearing in \cref{df:FSZ} naturally appear in these values. If a group is non-$FSZ$ at the value $z$ then the linear characters of $C_G(z)$ may be all we need to detect this with Schauenburg's criterion.  This cannot be a definitive test in full generality, however, as $C_G(z)$ may be a perfect group, or otherwise have too few linear characters (in an imprecise sense) to demonstrate the non-$FSZ$ property directly.

We next show that if a linear character suffices to demonstrate a group is non-$FSZ$ via Schauenburg's criterion, then we can obtain some information about the particular modules that would yield non-integer indicators.
\begin{thm}
  If $G$ is group with $z\in Z(G)$, $\chi\in\widehat{G}$, and $m\in\BN$ such that $\beta_m(\chi,z)\not\in\BQ$ then there exists $u\in G$ with $o(\chi(u))\not\in\{1,2,3,4,6\}$ such that $G$ is non-$FSZ_m$ over $u$.
\end{thm}
\begin{proof}
  Let assumptions be as in the statement. Consider the sum for $\beta_m(\chi,z)$ given by \cref{lem:lin-expand}. Since we are only interested in the irrationality of $\beta_m(\chi,z)$, those terms in the sum with $\chi(u)\in\BQ$ can clearly be ignored; in particular, those $u\in\ker(\chi)$ contribute a rational value to $\beta_m(\chi,z)$.  Moreover, since $\beta_m(\chi,z)$ is necessarily a real number, then by taking the real part of \cref{eq:beta-lin} and using the fact that fourth and sixth roots of unity have rational real parts, we see that those $u$ such that $\chi(u)$ has order in $\{1,2,3,4,6\}$ contribute a rational value to $\beta_m(\chi,z)$.

  Now let $\operatorname{cyc}(G)$ denote the set of all cyclic subgroups of $G$.  And for $X\in\operatorname{cyc}(G)$ let $\operatorname{gens}(X) = \{ x\in X \ | \ \cyc{x}=X\}$ be the set of generators of $X$.  Then
  \begin{align}
    \beta_m(\chi,z) = \sum_{X\in\operatorname{cyc}(G)}\sum_{u\in\operatorname{gens}(X)} |G_m(u,z)|\chi(u).
  \end{align}
  By \citep{PS:quasitensor} for any $n\in\BN$ with $(n,|G|)=1$ we have $|G_m(u^k,z)|=|G_m(u,z^n)|$ for $kn\equiv 1\bmod |G|$.  Now consider a fixed but otherwise arbitrary $X\in\operatorname{cyc}(G)$ and $u\in \operatorname{gens}(X)$. Then $\operatorname{gens}(X)=\{ u^k \ | \ (k,|G|)=1\}$.  It follows that if $|G_m(u,z)|=|G_m(u,z^n)|$ for all $n\in\BN$ with $(n,|G|)=1$ then for this $X,u$ we have
  \[ \sum_{v\in\operatorname{gens}(X)} |G_m(v,z)|\chi(v) = |G_m(u,z)|\sum_{\substack{k=1\\\gcd(k,o(u))=1}}^{o(u)}\chi(u)^k.\]
  This latter sum is always a rational value.  Ergo if $|G_m(u,z)|=|G_m(u,z^n)|$ for every $u\in G$ with $o(\chi(u))\not\in\{1,2,3,4,6\}$ and $(n,|G|)=1$, then $\beta_m(\chi,z)\in\BQ$, a contradiction.  This completes the proof.
\end{proof}
By using the correspondence theorem, we see that the preceding theorem equivalently says that there exists $u\in G$ such that $G$ is non-$FSZ_m$ over $u$ and that $u\not \in N$ for every normal subgroup $N\supseteq \ker(\chi)$ with $[N:\ker(\chi)]\in\{1,2,3,4,6\}$.

Since $p$-groups always admit non-trivial linear characters, these results provide a suggestive, but not definitive, procedure for testing $p$-groups for the $FSZ_{p^j}$ properties.  Namely, when checking if a $p$-group $P$ is non-$FSZ_{p^j}$ at $z\in P$, apply \cref{thm:peter} to the linear characters of $C_P(z)$ first.  The author has tested the non-$FSZ$ $p$-groups from \citep{IMM,K16:Sporadics} and has found this procedure successful for all of them.  The $p$-groups we consider in the remainder of the paper are also established as non-$FSZ$ in this fashion.

\begin{question}
  If $P$ is a $p$-group, is $P$ non-$FSZ_{p^j}$ at $z\in Z(P)$ if and only if $\beta_{p^j}(\chi,z)\not\in\BQ$ for some $\chi\in\widehat{P}$?
\end{question}

More generally, every (irreducible) character of $G$ has a kernel which is a normal subgroup of $G$.  Moreover, for any normal subgroup $N\subseteq G$, the quotient map lifts any irreducible character of $G/N$ to an irreducible character of $G$. By definition, $\beta_m(\chi,z)\not\in\BQ$ requires that $a\not\in\ker(\chi)$ for some $a\in G$ with $a^m=z$. It is therefore natural to identify the potential kernels $N$ that satisfy $a\not\in N$ for some $a\in G$ with $a^m=z$.
\begin{lem}\label{lem:kernels}
  Suppose $G$ is non-$FSZ_m$ at $z\in Z(G)$.  If $N$ is an $FSZ_m$ subgroup of $G$ then there exists $a\in G$ with $a^m=z$ and $a\not\in N$.
\end{lem}
\begin{proof}
  Let assumptions and notation be as in the statement.  By the assumption that $G$ is non-$FSZ_m$ at $z\in Z(G)$ there exists $u\in G$ such that $G_m(u,z)\neq\emptyset$.  If no such $a$ exists then $G_m(u,z)\neq\emptyset$ for some $u\in G$ implies that $u,z\in N$ and $G_m(u,z^n)=N_m(u,z^n)$ for all $n\in\BN$ with $(n,|G|)=1$.  Since $N$ is $FSZ_m$ this contradicts the assumption that $G$ is non-$FSZ_m$ at $z$.  Thus such an $a$ must exist.
\end{proof}
This says that if we suspect a group is non-$FSZ_m$ and want an irreducible representation for which \cref{thm:peter} is likely---in an imprecise sense---to establish the non-$FSZ_m$ property, we should consider the characters of $G/N$, where $N$ is an $FSZ_m$ normal subgroup of $G$.  This, and the author's prior computational experience with $\PSp_6(5)$ \citep{K16:Sporadics}, motivate our choices in \cref{sec:symplectic}.

\section{Basics of the Sylow subgroups}\label{sec:sylows}
We now review the essential facts about the (projective) symplectic groups and their Sylow subgroups in defining characteristic.  For the remainder of the paper, we let $p$ be a fixed but otherwise arbitrary odd prime.  Unless otherwise noted, $q$ will always be a power of $p$.

For any $n\in\BN$ we can define $\Sp_{2n}(q)$ as the group of isometries of a $2n$-dimensional $\BF_q$-vector space equipped with a symplectic form.  These can be described as $2n\times 2n$ matrices over $\BF_q$ decomposed into $n\times n$ blocks
\[ \begin{pmatrix}
  X& A\\ B&Y
\end{pmatrix}\]
satisfying
\begin{align}\label{eq:symp-rel}
\begin{pmatrix}
  X& A\\ B&Y
\end{pmatrix}\begin{pmatrix}
  Y^T&-A^T\\-B^T&X^T
\end{pmatrix} = I_{2n}.\end{align}
The center is of order two, generated by $-I_{2n}$.  The group \[\PSp_{2n}(q)=\Sp_{2n}(q)/Z(\Sp_{2n}(q))\] is simple when $(n,q)\neq (1,3)$ and is known as the projective symplectic group.  These two groups necessarily have isomorphic Sylow $p$-subgroups.  We also have that $\Sp_2(q)=\text{SL}_2(q)$.

We let $UT(n,q)$ denote the multiplicative group of all upper triangular $n\times n$ matrices over $\BF_q$ with all diagonal entries equal to $1$.  We say such matrices are (upper) unitriangular.
\begin{lem}\label{lem:L-order}
  Fix $n\in\BN$ and an odd prime $p$.  Let $q$ be any power of $p$, and set $t=\lceil\log_p(n)\rceil$.  Then $\exp(UT(n,q))=p^t$.
\end{lem}
\begin{proof}
  Any upper triangular $n\times n$ matrix over $\BF_q$ whose diagonal entries are all one is invertible, and is an element of $UT(n,q)$.  Since there are $n(n-1)/2$ entries above the diagonal, $|UT(n,q)| = q^{n(n-1)/2}$.  Therefore $UT(n,q)$ is a $p$-group, and so has exponent a power of $p$.  Moreover, there is a well-defined canonical Jordan form for a unitriangular matrix, which is again a unitriangular matrix.  It is then easy to see that the maximum possible order for an element of $UT(n,q)$ comes from an element with a single Jordan block, and such an element has order precisely $p^t$, with $t$ defined as in the statement.
\end{proof}

Next we consider the set of all matrices in $\text{GL}_{2n}(q)$ with the upper triangular block decomposition
\begin{align}\label{eq:blocks}
\begin{pmatrix}L^T&A\\0&L\inv\end{pmatrix}
\end{align}
with $L\in UT(n,q)$ being unitriangular and $A\in M_n(\BF_q)$ being any matrix such that $AL$ is symmetric: $(AL)^T=AL$. These matrices are well-known to give a Sylow $p$-subgroup of $\Sp_{2n}(q)$.  For the remainder of the paper we denote the above Sylow $p$-subgroup by $P$, where $n$ and $q$ should be clear from the context, or otherwise arbitrary.

\begin{prop}\label{prop:abelianization}
  Let $\BF_q^+$ denote the additive group of $\BF_q$.  Then the map $\kappa\colon P\to (\BF_q^+)^n$ given by
  \[ \kappa(X) = (A_{1,1},L_{1,2},L_{2,3},...,L_{n-1,n}),\]
  where $X\in P$ has the block decomposition in terms of $A,L$ given in \cref{eq:blocks}, is a surjective group homomorphism.
\end{prop}
\begin{proof}
  That $\kappa$ is surjective is immediate.  So let $X,Y\in P$ have block decompositions
  \begin{align*}
    X &= \begin{pmatrix}
    L^T & A\\ 0 & L\inv
    \end{pmatrix},\\
    Y &=\begin{pmatrix}
      M^T & B\\ 0 & M\inv
    \end{pmatrix}.
  \end{align*}
  Then
  \begin{align*}
    XY = \begin{pmatrix}
      L^T M^T & L^T B +A M\inv\\
      0 & L\inv M\inv
    \end{pmatrix}.
  \end{align*}
  Since $M,L$ are unitriangular it readily follows that $(ML)_{i,i+1} = \sum_k M_{i,k} L_{k,i+1} = M_{i,i+1}+L_{i,i+1}$ for all $1\leq i < n$.  Moreover, again using that $M,L$ are unitriangular we have that
  \[ (L^T B + A M\inv)_{1,1} = (L^T B)_{1,1}+(A M\inv)_{1,1} = B_{1,1}+A_{1,1}.\]
  Thus $\kappa$ is a group homomorphism, as desired, and this completes the proof.
\end{proof}

\begin{cor}\label{cor:P-character}
  Given any non-trivial character $\lambda\colon \BF_q^+\to \BC^\times$, the map $\xi_\lambda\colon P \to \BC^\times$ given by $\xi_\lambda(X) = \lambda(A_{1,1})$, where $X$ has the block decomposition given in \cref{eq:blocks}, is a non-trivial linear character of $P$.
\end{cor}
\begin{proof}
  This is an immediate consequence of \cref{prop:abelianization}.
\end{proof}

We denote the elementary matrices by $E_{i,j}$, which is the matrix with a $1$ in position $(i,j)$ and zeroes everywhere else.  The dimension of $E_{i,j}$ is not particularly important, but we will implicitly assume that $E_{i,j}$ is a square matrix of suitable dimensions---usually an $n\times n$ or $2n\times 2n$ matrix when talking about $\Sp_{2n}(q)$---wherever it appears.

We will have particular need for computing $p$-th powers of arbitrary matrices in $P$.  To this end, we first note the following.
\begin{lem}\label{lem:p-roots}
  Let
  \[ M=\begin{pmatrix} L^T& A\\0&L\inv\end{pmatrix} \in P.\]
  Then for any $j\in\BN$
  \begin{align}\label{eq:powers} M^j = \begin{pmatrix} (L^j)^T& \Big(\sum_{m=0}^{j-1} (L^m)^T A L^m\Big)L^{1-j}\\
  0&L^{-j}\end{pmatrix}.\end{align}
  As special cases we have the following.
  \begin{enumerate}
    \item If $L$ has order $p^k$ then
  \[ M^{p^k} = \begin{pmatrix} I_n& \Big(\sum_{m=0}^{p^{k}-1} (L^m)^T A L^m\Big)L\\
  0&I_n\end{pmatrix}.\]
    \item If $L$ has order $p^k$ then $M$ has order either $p^k$ or $p^{k+1}$.
  \end{enumerate}
\end{lem}
\begin{proof}
  The desired formula for $M^j$ is an easy induction, and the special cases are immediate consequences.
\end{proof}

The summation appearing in \cref{eq:powers} will appear several times in the remainder of the paper when $j$ is a power of $p$.  So for ease of notation we have the following.
\begin{df}\label{df:YL}
For a given unitriangular matrix $L\in UT(n,q)$ and $k\in\BN$, we define a map
\begin{gather*}
Y_{L,k}\colon M_n(\BF_q)\to M_n(\BF_q)\\
A\mapsto \sum_{m=0}^{p^k-1} (L^m)^T A L^m.
\end{gather*}
\end{df}
This is clearly an $\BF_q$-linear map and so is completely determined by its values on the elementary matrices.  Note that $Y_{L,k}$ is the zero map whenever the order of $L$ is less than $p^k$.

Now for any upper triangular matrix $L\in UT(n,q)$ we can write $L= \sum_{i\leq j} l_{i,j}E_{i,j}$ for some scalars $l_{i,j}\in\BF_q$.  Then for any $m,i,j\in\BN$ with $i\leq j\leq n$ we have that $L^m$ is also upper-triangular with entries
\begin{align}\label{eq:L-powers}
  (L^m)_{i,j} = \sum_{i= i_0\leq i_1\leq\cdots \leq i_m= j}\prod_{a=1}^{m} l_{i_{a-1},i_a}.
\end{align}
Note that the sum here is over all non-decreasing sequences of positive integers of length $m+1$ that start at $i$ and end at $j$.  

\begin{df}\label{df:upsilon}
  Given a unitriangular matrix $L\in \GL{n}{\BF_q}$ with $2n\geq p+1$ and $k\in\BN$ such that $p^k\leq 2n-1$ we define the scalar
  \[ \Upsilon(L,k) = \prod_{i=1}^{(p^k-1)/2} l_{i,i+1}^2,\]
  which is the product of the squares of the first $(p^k-1)/2$ entries immediately above the diagonal.
\end{df}
Note that we always have $\Upsilon(L,k)\in QR(q)$. Our goal is to connect $Y(L,k)$, $\Upsilon(L,k)$, and \cref{lem:p-roots}.  To this end, we recall a few results on congruences modulo a prime.

\begin{lem}\label{lem:poly-vanish}
  Let $p$ be a prime. For any $k\in\BZ$
  \begin{align*}
    \sum_{i=1}^{p-1} i^k \equiv \begin{cases}
      0\bmod p,& p-1\nmid k\\
      -1\bmod p,&p-1\mid k.
    \end{cases}
  \end{align*}
\end{lem}
\begin{proof}
  By Fermat's little theorem, we have that $a^{k}\equiv 1\bmod p$ for all $a\not\equiv 0\bmod p$ if and only if $p-1$ divides $k$.  This gives the $p-1\mid k$ case. So we may suppose that there exists $a\not\equiv0\bmod p$ with $a^k\not\equiv 1\bmod p$.  Since left multiplication by $a$ is a bijection on $\BZ_p^\times$, we have
  \[ \sum_{i=1}^{p-1} i^k \equiv \sum_{i=1}^{p-1} (ai)^k = a^k \sum_{i=1}^{p-1} i^k\bmod p,\]
  and thus the summation must be divisible by $p$, as desired.
\end{proof}

For non-negative integers $a,b$ we define the binomial coefficients \[\binom{a}{b} = \frac{a!}{b!(a-b)!}.\] We adopt the conventions that $0!=1$ and $\binom{a}{b}=0$ whenever $b>a$.
\begin{lem}\label{lem:binom-vanish}
  Let $p$ be an odd prime, $j\in\BN$, and $k,l\in\BZ$ be such that $0\leq k,l\leq (p^j-1)/2$.  If $k+l<p^j-1$ then
  \[ \sum_{m=0}^{p^j-1}\binom{m}{k}\binom{m}{l} \equiv 0 \bmod p.\]
  Else, when $k=l=(p^j-1)/2$ we have
  \[ \sum_{m=0}^{p^j-1} \binom{m}{(p^j-1)/2}^2 \equiv (-1)^{j(p-1)/2}\bmod p.\]
\end{lem}
\begin{proof}
  Let assumptions and notation be as in the statement. We proceed by induction on $j$.

  For $j=1$, let $k,l\in\BZ$ with $0\leq k,l<p-1$.  The result is trivial if $k=l=0$. So assuming $k+l>0$ we have
  \[ \sum_{m=0}^{p-1} \binom{m}{k}\binom{m}{l} = \frac{1}{k!l!} \sum_{m=1}^{p-1} \prod_{a=0}^{k-1}(m-a)\prod_{b=0}^{l-1}(m-b).\]
  Note that $k!$ and $l!$ are units modulo $p$, so this expression also makes sense modulo $p$. The product inside the summation is a polynomial in $m$ of degree $k+l$ with a constant term of $0$.  When $k+l<p-1$, then by \cref{lem:poly-vanish} the summation vanishes modulo $p$, as desired. On the other hand, when $k+l=p-1$, then also by \cref{lem:poly-vanish} the summation is equal to $-1/(k!l!)$.  A consequence of Wilson's lemma is that $(((p-1)/2)!)^2\equiv (-1)^{(p+1)/2}\bmod p$, which gives the $k=l=(p-1)/2$ case. A simple induction then permits the evaluation of $k!l!$ for any $k,l$ with $k+l=p-1$, but since the statement of the theorem requires us to only consider the cases $0\leq k,l\leq (p-1)/2$, we have completed the case $j=1$.

  Now let $1<j\in\BN$ and suppose the result holds for all smaller values of $j$. Let $k,l\in\BZ$ be such that $0\leq k,l\leq (p^{j}-1)/2$.  The result is again trivial when $k=l=0$, so we may suppose that $k+l>0$.  Expanding in base $p$, by assumptions we may write
  \begin{gather*}
    \frac{p^j-1}{2} = \sum_{y=0}^{j-1} \frac{p-1}{2} p^y\\
    k = \alpha p^{j-1} + k'\\
    l = \beta p^{j-1} + l'
  \end{gather*}
  for some integers $0\leq k',l'<p^{j-1}$  and $0\leq\alpha,\beta\leq (p-1)/2$.  Then by also expanding $m=n p^{j-1}+m'$ in base $p$ and applying Lucas's theorem (see \cite[Theorem 1]{Fine:Binomial}) we have
  \begin{align}\label{eq:lucas}
    \sum_{m=0}^{p^{j}-1} \binom{m}{k}\binom{m}{l} \equiv \sum_{m'=0}^{p^{j-1}-1}\binom{m'}{k'}\binom{m'}{l'}\Big(\sum_{n=0}^{p-1} \binom{n}{\alpha}\binom{n}{\beta}\Big) \bmod p.
  \end{align}
  By the base case the inner summation vanishes whenever $0\leq \alpha,\beta\leq (p-1)/2$ and $\alpha+\beta<p-1$.  In the remaining case of $\alpha=\beta=(p-1)/2$, then by the base case again the inner summation is congruent to $(-1)^{(p-1)/2}$.  In this case the upper bound on $k,l$ then forces $k',l'\leq (p^{j-1}-1)/2$, so we may then apply the inductive hypothesis to the remaining summation.

  This completes the proof.
\end{proof}

We can now connect $Y_{L,k}$ and $\Upsilon(L,k)$.
\begin{thm}\label{thm:upsilon}
  Fix an odd prime $p$ and $n\in\BN$ such that $r=\lceil \log_p(2n)\rceil\geq 2$, and let $L\in UT(n,q)$ be unitriangular. Then for any $a,b,s,t,y\in \BN$ with $y<r$; $s,t\leq n$; and $a,b\leq (p^y+1)/2$ we have
  \begin{align}\label{eq:upsilon-1}
    Y_{L,y}(E_{s,t})_{a,b} = (-1)^{y(p-1)/2}\delta_{a,b,(p^y+1)/2}\delta_{s,t,1}\Upsilon(L,y).
  \end{align}
\end{thm}
\begin{proof}
  We fix $L,y,a,b,r,s,t$ as in the statement and let $Y$ stand for $Y_{L,y}$.

  We observe that from \cref{eq:L-powers} it follows for any $1\leq i,j\leq n$ that
  \begin{align}\label{eq:y-st-entries}
    Y(E_{s,t})_{i,j} &= \sum_{m=1}^{p^y}\sum \sum \prod_{c,d=1}^m l_{u_{c-1},u_c}l_{v_{d-1},v_d},
  \end{align}
  where the unlabeled double summation is over all pairs of non-decreasing sequences $\{u_c\}$ and $\{v_d\}$ of length $m+1$ satisfying
  \begin{align*}
    u_0 &=t,\qquad u_m =j;\\
    v_0 &=s,\qquad v_m =i.
  \end{align*}
  Note that if $t>j$ or $s>i$ then $Y(E_{s,t})_{i,j}=0$ as desired.  So we may suppose that $s\leq i$ and $t\leq j$.

  We consider the super-diagonal elements $l_{i,j}$ of $L$ as indeterminates for the remainder of the proof.

  Each distinct product of indeterminates appearing in \cref{eq:y-st-entries} can be specified by a pair of strictly increasing sequences, where a constant length one sequence is trivially strictly increasing: one beginning at $t$ and ending at $j$, and the other beginning at $s$ and ending at $i$. Note that by assumptions on $a,b$ the maximum length of any such sequence appearing in $Y(E_{s,t})_{a,b}$ is $(p^y+1)/2$, and strictly increasing sequences with this maximal length are uniquely determined.  To compute the coefficient on such a product of indeterminates we need to then count, for each $1\leq m\leq p^y$, the number of non-decreasing sequences containing each such choice as its maximal strictly increasing subsequence.  Note that there will, in general, be multiple such choices of pairs of strictly increasing or constant sequences that provide the given product.  You can often just switch the order they are selected in, namely.  However, to show that the necessary coefficients vanish modulo $p$ it suffices to show that the contribution from each such pair of strictly increasing sequences, not both of maximum possible length, vanishes modulo $p$.

  For any given strictly increasing sequence of length $1\leq k_1\leq (p^y+1)/2$, the number of non-decreasing sequences of length $m+1$ containing it as their maximal strictly increasing subsequence is precisely the number of compositions (see \citep[Chapter 1, Section 1]{HM:Combinatorics}) of $m+1$ of length $k_1$. Here, each term in the composition tells us how many times the corresponding entry in the given maximal subsequence is repeated.  Therefore, for a given pair of strictly increasing sequences of lengths $k_1$ and $k_2$ respectively, by \citep[Theorem 1.3]{HM:Combinatorics} the coefficient on the product of the indeterminates they determine is precisely
  \begin{align*}
    \sum_{m=\max(k_1,k_2)-1}^{p^y-1}\binom{m}{k_1-1} \binom{m}{k_2-1} &= \sum_{m=1}^{p^y-1}\binom{m}{k_1-1} \binom{m}{k_2-1}.
  \end{align*}
  Since the product of indeterminates determined by a pair of maximal length strictly increasing sequences is precisely $\Upsilon(L,y)$, we can then apply \cref{lem:binom-vanish} to complete the proof.
\end{proof}
To provide a clarifying visual, the result says that the upper left $(p^y+1)/2\times (p^y+1)/2$ block of $Y_{L,y}(A)$ has at most one non-zero entry for any $A$, and this entry occurs in the bottom right corner.

Finally, we then connect this back to \cref{lem:p-roots}.
\begin{thm}\label{thm:block-comp}
  Let $1\neq Y,X\in P$ have block decompositions
  \begin{gather*}
    X = \begin{pmatrix}
      L^T&A\\ 0 &L\inv
    \end{pmatrix},\\
    Y = \begin{pmatrix}
      K^T & B\\ 0&K\inv
    \end{pmatrix}.
  \end{gather*}
  Suppose for some $k\in\BN$ we have $X^{p^k}=Y$.  Set $s=(p^k+1)/2$.  Then the upper-left $s\times s$ block of $B$ is equal to $(-1)^{k(p-1)/2}A_{1,1}\Upsilon(L,k)E_{s,s}$.
\end{thm}
\begin{proof}
  Let notation be as in the statement. The assumption $X^{p^k}=Y\neq 1$ ensures that the statement is well-defined: by \cref{lem:L-order} the dimension $2n$ necessarily satisfies $2n> p^k$. By \cref{lem:p-roots} the upper-left $s\times s$ block of $B$ is equal to the upper-left $s\times s$ block of $Y_{L,k}(A) L^{1-p^k}$.  By \cref{thm:upsilon}, $Y_{L,k}(A)$ has its upper-left $s\times s$ block equal to $(-1)^{k(p-1)/2}A_{1,1}\Upsilon(L,k)E_{s,s}$. Since $L^{1-p^k}$ is unitriangular it follows that the upper-left $s\times s$ block of $B$ is precisely $(-1)^{k(p-1)/2}A_{1,1}\Upsilon(L,k)E_{s,s}$, as desired.
\end{proof}

\begin{thm}\label{thm:all-res}
  Fix an odd prime $p$ and $n\in\BN$ such that $r=\lceil\log_p(2n)\rceil\geq 2$. Let $P$ be the Sylow $p$-subgroup of $\Sp_{2n}(q)$, where $q$ is some power of $p$, defined above. For each $1\leq j<r$ let $E(j)$ denote the $n\times n$ elementary matrix with a 1 in the $((p^j+1)/2,(p^j+1)/2)$ position.  Then for all $1\leq j < r$, $0\neq x\in QR(q)$, and
  \begin{align}\label{eq:g-def-0}
    g_j = \begin{pmatrix} I_n & (-1)^{j(p-1)/2}E(j)\\ 0 & I_n\end{pmatrix}\in P
  \end{align}
  there exists $X\in P$ with
  \begin{gather*}
    X = \begin{pmatrix}
      L^T & A\\ 0 & L\inv
    \end{pmatrix}
  \end{gather*}
  such that $X^{p^j}=g_j$ and $A_{1,1}=x$.
\end{thm}
\begin{proof}
  Consider first the special case $2n=p^j+1$.  From \cref{lem:p-roots,thm:block-comp} it follows that $X^{p^j}=g_j$ if and only if
  \begin{gather*}
    A_{1,1}\Upsilon(L,j) = 1.
  \end{gather*}
  In particular, some solution to $X^{p^j}=g_j$ exists. For any $0\neq x\in QR(q)$, we may change the super-diagonal entries of $L$ to yield a matrix $L'$ that satisfies $\Upsilon(L',j)=x\inv$ and then define $A'$ to be any suitable matrix with $A'_{1,1}=x$.  This gives the desired result when $2n=p^j+1$.

  For $2n>p^j+1$, the Sylow $p$-subgroup of $\Sp_{p^j+1}(q)$ embeds into the Sylow $p$-subgroup of $\Sp_{2n}(q)$ by
  \[ \begin{pmatrix}
    L^T& A\\ 0 & L\inv
  \end{pmatrix} \mapsto \begin{pmatrix}
    \begin{pmatrix}
      L^T&0\\0&1
    \end{pmatrix} & \begin{pmatrix}
      A & 0\\ 0 & 0
    \end{pmatrix}\\
    0 & \begin{pmatrix}
      L\inv & 0\\ 0 &1
    \end{pmatrix}
  \end{pmatrix}.\]
  This embedding maps any solution from the case $2n=p^j+1$ to a solution for the case $2n>p^j+1$, and so completes the proof.
\end{proof}

%
%
\section{Sylow subgroups via characters}\label{sec:sylow-characters}

We now have all of the ingredients necessary to establish the non-$FSZ$ properties for Sylow subgroups with suitable $p,q,n$.
\begin{thm}\label{thm:sylow-fsz}
  Let $p>3$ be an odd prime with $p\equiv 1\bmod 4$, and $q$ any odd power of $p$.  Then for any $j\in\BN$ the Sylow $p$-subgroups of $\Sp_{p^j+1}(q)$ and $\PSp_{p^j+1}(q)$ are non-$FSZ_{p^j}$.
\end{thm}
\begin{proof}
  Let $p$ be an odd prime and $q$ a power of $p$, fix $j\in\BN$, and define $n\in\BN$ by $2n=p^j+1$. The Sylow subgroups in question are isomorphic, so we need only consider the Sylow $p$-subgroup of $\Sp_{2n}(q)$.  We work in the standard Sylow $p$-subgroup given in \cref{sec:sylows}.

  We define $\sigma_j = (-1)^{(p^j-1)/2} = (-1)^{j(p-1)/2}$.  Note that $\sigma_j=1$ whenever $p\equiv 1\bmod 4$, and when $p\equiv 3\bmod 4$ then $\sigma_j = (-1)^j$.  We define $g_j\in P$ by \cref{eq:g-def-0}, so that
  \begin{align}\label{eq:g-def}
    g_j = \begin{pmatrix}
      1 & (-1)^{j(p-1)/2}E_{n,n}\\
      0&1
    \end{pmatrix}.
  \end{align}

  By \cref{thm:all-res,thm:block-comp} solutions to $X^{p^j}=g_j$ exist and are completely determined by the condition $A_{1,1}\Upsilon(L,j) = 1$.  Of necessity $0\neq\Upsilon(L,j)\in QR(q)$.  Moreover, every non-zero quadratic residue is achieved as the value $\Upsilon(L,j)$ an equal number of times among the solutions to $X^{p^j}=g_j$.

  Next let $\psi=\psi_{p^j,g_j}$ be defined as in \cref{df:psi-beta}.  For any non-trivial character $\lambda$ of $\BF_q^+$, we define the linear character $\xi_\lambda$ of $P$ as in \cref{cor:P-character}.  Then we have
  \[ \ip{\xi_\lambda}{\psi} = \alpha \sum_{x\in QR(q)}\lambda(x)\]
  for some $0\neq\alpha\in\BQ$.  It follows that
  \begin{equation*}
    \beta_{p^j}(\xi_\lambda,g_j) \sim_\BQ \sum_{x,y\in QR(q)}\lambda(x-y).
  \end{equation*}

  By \cref{lempart:residues-3} every element of $\BF_q^\times$ is expressible as a difference of quadratic residues the same number of ways if and only if $-1\not\in QR$.  In this case we immediately conclude that $\beta_{p^j}(\xi_{\lambda},g_j)\in \BQ$, and Schauenburg's criterion is inconclusive.  On the other hand, when $-1\in QR$, we have that either $p\equiv 1\bmod 4$ or that $q$ is an even power of $p$, and we may write
  \begin{equation}\label{eq:sylow-beta-2}
    \beta_{p^j}(\xi_{\lambda},g_j) \sim_\BQ \sum_{\substack{x,y,z\in QR(q)\\0\neq z=x-y}}\lambda(z).
  \end{equation}
  Since this sum is now only over $0\neq z\in QR(q)$, and $-1\in QR(q)$, by \cref{lempart:residues-3} we have
  \begin{equation}\label{eq:sylow-beta-3}
    \beta_{p^j}(\xi_{\lambda},g_j)\sim_\BQ\sum_{z\in QR(q)}\lambda(z).
  \end{equation}

  Fix some primitive $p$-th root of unity $\mu_p\in\BC$.   We have that $\sum_{i=1}^{p-1}\mu_p^i=-1$ and that every proper subset of $\{\mu_p^i\}_{i=0}^{p-1}$ is linearly independent over $\BQ$.  So we conclude that the sum in \cref{eq:sylow-beta-3} gives a rational value if and only if every value of $\mu_p^i$ for $0<i<p$ appears equally often as the image under $\lambda$ of an element of $QR(q)$.  By the last two parts of \cref{thm:char-residues} this is equivalent to $q$ being an even power of $p$.  Therefore by Schauenburg's criterion we obtain the desired result.
\end{proof}

The specific choice of dimension $2n=p^j+1$ was necessary in order to ensure that each non-zero quadratic residue appears an equal number of times as the value of $A_{1,1}$ in the solutions to $X^{p^j}=g_j$. Without this fact, it is conceivably possible (in the sense that the author was not able to rule it out) that some quadratic residues appear more often than others, and that things might manage to perfectly balance each other out and therefore yield a rational $\beta$ value.

\begin{example}
    Using \citet{GAP4.8.4}, the author was able to verify that the Sylow $5$-subgroup $P$ of $\Sp_{8}(5)$ is non-$FSZ_5$ at an element $z\in Z(P)$.  In particular, there is a linear character $\chi\in\widehat{P}$ with $\beta_5(\chi,z)\not\in\BQ$.  The GAP routines of \citep{PS16} were unable to handle this group on the author's computer.  Instead the author computed only the linear characters and a corresponding $\beta$ value.

    On the other hand, the GAP routines of \citep{PS16} are sufficient (with many days of waiting) to show that the Sylow 7-subgroup of $\Sp_8(7)$ is $FSZ$.
\end{example}
As such the preceding result is not the best possible, and we suspect that the result should hold in much greater generality.
\begin{question}\label{q:sylow}
  Let $P$ be the Sylow $p$-subgroup of $\Sp_{2n}(q)$.  Set $r=\lceil\log_p(2n)\rceil$.  Which, if any, of the following are equivalent?
  \begin{enumerate}
    \item $P$ is non-$FSZ_{p^j}$ for some $1\leq j<r$;
    \item $P$ is non-$FSZ_{p^j}$ for all $1\leq j<r$;
    \item $P$ is non-$FSZ$ at some $z\in Z(P)$;
    \item $r>1$, $q$ is an odd power of $p$, and $p\equiv 1\bmod 4$.
  \end{enumerate}
\end{question}

\section{Sylow subgroups via counting}\label{sec:sylow-counting}
Our goal for this section is to provide a second, combinatorial proof of \cref{thm:sylow-fsz} using \cref{df:FSZ}.  The proof naturally starts off the same.

\begin{proof}
    Let assumptions and notation be as in the statement of \cref{thm:sylow-fsz}. So, in particular, we have $p\equiv 1\bmod 4$ and that $q$ is an odd power of $p$.

    We define $g_j\in P$ by \cref{eq:g-def}, where by assumptions on $p$ the term $(-1)^{j(p-1)/2}$ is always $1$.
    Fix any $d\in\BN$ with $(d,|G|)=1$.  Observe that
    \[ g_j^d = I_{2n}+ d E_{n,2n} = \begin{pmatrix}
      1 & d E_{n,n}\\
      0&1
    \end{pmatrix}.\]

    By \cref{thm:all-res,thm:block-comp} and their proofs solutions to $X^{p^j}=g_j^d$ with $X\in P$ exist and are completely determined by the condition $A_{1,1}\Upsilon(L,j) = d$.  Of necessity $0\neq\Upsilon(L,j)\in QR(q)$ and $\legendre{A_{1,1}}{q}=\legendre{d}{q}$.

    Now define $U\in P$ by
  \begin{align}\label{eq:U-def}
    U = \begin{pmatrix}
    I_n+E_{2,1}& E_{1,1}\\
    0& I_n-E_{1,2}
  \end{pmatrix}.
  \end{align}
  Applying \cref{thm:all-res,thm:block-comp} and their proofs again we see that the elements $X\in P$ with $(XU)^{p^j}=g^d$ are precisely those matrices satisfying
  \[ (A_{1,1}+1)(l_{1,2}+1)^2 \prod_{i=2}^{n-1} l_{i,i+1}^2=d.\]

  It suffices to determine conditions that guarantee that there is a different number of solutions to the equations $X^{p^j}=(XU)^{p^j}=g_j$ and $X^{p^j}=(XU)^{p^j}=g_j^{d}$ for some integer $d\in\{2,...,p-1\}$. We let $d_1=1$ and fix $d_2\in\{2,...,p-1\}$.

  Now $X^{p^j}=(XU)^{p^j}=g^{d_k}$ for $k=1$ or $k=2$ if and only if
  \[ A_{1,1}l_{1,2}^2 \prod_{i=2}^{n-1}l_{i,i+1}^2 = (A_{1,1}+1)(l_{1,2}+1)^2 \prod_{i=2}^{n-1} l_{i,i+1}^2 = d_k.\]
  We observe that $\prod_{i=2}^{n-1}l_{i,i+1}^2$ is a non-zero quadratic residue in $\BF_q$, and moreover that every non-zero quadratic residue can be obtained this way an equal number of times among solutions to either equation.

  For $k=1,2$ and $0\neq y\in QR(q)$, we define the two variable polynomials
  \begin{align}
    f_{k,y}(a,b) &= ab^2-d_k y\\
    g_{k,y}(a,b) &= (a+1)(b+1)^2 - d_k y.
  \end{align}
  The result follows if we can show that the number of pairs $(a,b)$ such that there exists $0\neq y\in QR$ with $f_{k,y}(a,b)=0=g_{k,y}(a,b)$ is different for $k=1,2$ for some choice of $d_2$. Indeed, we claim that the number of such pairs depends on whether or not $d_2$ is a quadratic residue.  By assumptions we note that any such pairs have $a,b\not\in\{0,-1\}$.

  Now $f_{k,y}(a,b)=g_{k,y}(a,b)$ if and only if
  \[\frac{a}{a+1} = \Big(\frac{b+1}{b}\Big)^2.\]
  Let $r_1= ((b+1)/b)^2$, which we observe is a non-zero quadratic residue.  We can solve for $a$ to get
  \[ a = \frac{r_1}{1-r_1}.\]
  Note that by definition $r_1\neq 1,0$, which then implies $a\neq 0,-1$.  Now set $r_2 = y/b^2$, and observe that $r_2$ is also a non-zero quadratic residue.  The equation $f_{k,y}(a,b)=0$ can then be rewritten as \[r_1=d_k r_2(1-r_1).\]

  So if $d_k\in QR(q)$ we must have that $1+(-r_1)$, which is a sum of quadratic residues, is itself a quadratic residue.  By \cref{lem:residues} there are $(q+3)/4$ such values of $r_1$, but since $r_1\neq 0,1$, we have $(q-5)/4$ choices of $r_1$.  For such a choice, $r_2$ is uniquely determined, and then since $r_2=y/b^2$ we see that there are precisely two pairs $(a,b)$ yielding the pair of quadratic residues $(r_1,r_2)$ for some $0\neq y\in QR$. Thus there are $(q-5)/2$ solutions for $d_k\in QR(q)$.

  On the other hand, if $d_k\not\in QR(q)$ we must have that $1+(-r_1)$, which is a sum of quadratic residues, is itself not a quadratic residue.  Applying \cref{lem:residues} again we see that there are $(q-1)/4$ such choices for $r_1$, which again uniquely determines a quadratic residue $r_2$, and that there are precisely two pairs $(a,b)$ yielding the pair of quadratic residues $(r_1,r_2)$ for some $0\neq y\in QR$.  Thus there are \[(q-1)/2 = (q-5)/2 +2\] solutions for $d_j\not\in QR(q)$.

  Assumptions on $p,q$ guarantee that $QR(q)$ does not contain the entire prime sub-field.  So by taking $d_2\in\BZ_p\setminus QR(q)$ and applying \cref{df:FSZ} to the sets $G_{p^j}(U,g_j)$ and $G_{p^j}(U,g_j^{d_2})$ we obtain the desired result.
\end{proof}

\begin{example}
  As a consequence of the result, the Sylow $5$-subgroup $P$ of $G=\PSp_6(5)$ is non-$FSZ_5$, and exactly one of the two sets $P_5(U,g_1), P_5(U,g_1^2)$ is empty.
\end{example}

\section{The full group and centralizer}\label{sec:symplectic}
Throughout this section we fix $j\in\BN$, define $n\in\BN$ by $2n=p^j+1$, and set $G=\Sp_{2n}(q)$, $H=\PSp_{2n}(q)$. We also define $g_j\in G$ by \cref{eq:g-def}, and let $[g_j]\in H$ be the image of $g_j$ under the quotient map.
\begin{lem}\label{lem:cent-1}
   With notation as above, $C_G(g_j)$ consists of those $M\in G$ with the block form
   \begin{equation}\label{eq:cent-form-1}
    \begin{pmatrix}
      \begin{pmatrix}
        X' & 0\\
        x_2& \Lambda
      \end{pmatrix} &
      \begin{pmatrix}
        A' & a_1\\
        a_2 & a_3
      \end{pmatrix}\\
      \begin{pmatrix}
        B'&0\\
        0&0
      \end{pmatrix} &
      \begin{pmatrix}
        Y' & y_2\\
        0 & \Lambda
      \end{pmatrix}
    \end{pmatrix},
  \end{equation}
  where $\Lambda\in\kk^\times$ is a scalar and $X'$ is an $(n-1)\times (n-1)$ matrix.
\end{lem}
\begin{proof}
  We have that $M\in C_G(g_j)$ if and only if $Mg_j = g_jM$.  Since $g_j= I_{2n}+(-1)^{j(p-1)/2}E_{n,2n}$, this is equivalent to $E_{n,2n}M = M E_{n,2n}$.  We have that $M E_{i,j}$ is the matrix whose $j$-th column is the $i$-th column of $M$ and all other entries zero.  Similarly, $E_{i,j} M$ is the matrix whose $i$-th row is the $j$-th row of $M$, and all other entries zero.  The desired result then follows.
\end{proof}
As in the remarks following \cref{lem:kernels}, we seek an $FSZ$ normal subgroup of $C_G(g_j)$ whose corresponding quotient group has irreducible representations that are relatively nice to compute with.  To this end we have the following.
\begin{prop}\label{prop:quot}
    There is a surjective group homomorphism \[\pi\colon C_G(g_j)\to \Sp_{p^j-1}(q)\times \{\pm 1\},\] where $\{\pm 1\}\cong \BZ_2$ in the usual fashion.

    Explicitly, writing $M\in C_G(g_j)$ as in \cref{lem:cent-1} we have
    \[ \pi(M) = (\begin{pmatrix}
      X' & A'\\
      B' & Y'
    \end{pmatrix}, \Lambda).\]
    In particular, as a scalar $\Lambda=\pm 1$.
\end{prop}
\begin{proof}
  Let $t,h\in C_G(g_j)$ have the block decompositions

  \begin{align*}
    t&=\begin{pmatrix}
      \begin{pmatrix}
        X' & 0\\
        x_2& \Lambda
      \end{pmatrix} &
      \begin{pmatrix}
        A' & a_1\\
        a_2 & a_3
      \end{pmatrix}\\
      \begin{pmatrix}
        B'&0\\
        0&0
      \end{pmatrix} &
      \begin{pmatrix}
        Y' & y_2\\
        0 & \Lambda
      \end{pmatrix}
    \end{pmatrix},\\
    h & =
    \begin{pmatrix}
      \begin{pmatrix}
        M & 0\\
        m_2& \Omega
      \end{pmatrix} &
      \begin{pmatrix}
        C & c_1\\
        c_2 & c_3
      \end{pmatrix}\\
      \begin{pmatrix}
        D&0\\
        0&0
      \end{pmatrix} &
      \begin{pmatrix}
        N & n_2\\
        0 & \Omega
      \end{pmatrix}
    \end{pmatrix}.
  \end{align*}
  A straightforward calculation shows that $th$ has the form
  \begin{equation}
    th =
    \begin{pmatrix}
      \begin{pmatrix}
        X'M+A'D & 0\\
        *& \Lambda\Omega
      \end{pmatrix} &
      \begin{pmatrix}
        X'C+A'N & *\\
        * & *
      \end{pmatrix}\\
      \begin{pmatrix}
        B'M+Y'D&0\\
        0&0
      \end{pmatrix} &
      \begin{pmatrix}
        B'C+Y'N & *\\
        0 & \Lambda\Omega
      \end{pmatrix}
    \end{pmatrix},
  \end{equation}
  where the asterisks denote entries we do not need to compute to establish the necessary results.

  Considering the special case $h=t\inv$, by \cref{eq:symp-rel} we see that
  \begin{gather*}
    \begin{pmatrix}
      X' & A'\\
      B' & Y'
    \end{pmatrix}\in\Sp_{p^j-1}(q)\\
    \Lambda \Lambda^T = 1.
  \end{gather*}
  Since $\Lambda$ is a scalar, it follows that $\pi$ is a well-defined group homomorphism. The only claim left to prove is that $\pi$ is surjective.

  Given any
  \begin{gather*}
    \begin{pmatrix}
      X' & A'\\
      B' & Y'
    \end{pmatrix}\in\Sp_{p^j-1}(q)\\
    \Lambda\in\{\pm 1\},
  \end{gather*}
  an easy check of \cref{eq:symp-rel} shows that
  \begin{gather*}
    t=\begin{pmatrix}
      \begin{pmatrix}
        X' & 0\\
        0& \Lambda
      \end{pmatrix} &
      \begin{pmatrix}
        A' & 0\\
        0 & 0
      \end{pmatrix}\\
      \begin{pmatrix}
        B'&0\\
        0&0
      \end{pmatrix} &
      \begin{pmatrix}
        Y' & 0\\
        0 & \Lambda
      \end{pmatrix}
      \end{pmatrix}
  \end{gather*}
  is an element of $G$, and so in $C_G(g_j)$ by \cref{lem:cent-1}, and satisfies
  \[ \pi(t) = (\begin{pmatrix}
      X' & A'\\
      B' & Y'
    \end{pmatrix}, \Lambda).\]
  Therefore $\pi$ is surjective as desired, and this completes the proof.
\end{proof}
At this point we do not yet need that $2n=p^j+1$, only that $2n>p^j$.  In the subsequent, much as with the Sylow subgroups before, the primary reason we restrict to $2n=p^j+1$ is to make it easy to predict and control certain properties of solutions to $a^{p^j}=g_j$.

The following combined with \cref{ex:IMM} shows that the kernel of $\pi$ and the preimage of $\{\pm 1\}$ under $\pi$ are necessarily $FSZ$.
\begin{cor}\label{cor:ker-prop}
  The kernel of the group homomorphism $\pi$ has exponent $p$ and contains $g_j$.
\end{cor}
\begin{proof}
    By the definition of $\pi$ and \cref{lem:cent-1} we see that the kernel consists precisely of those $M\in C_G(g_j)$ (indeed, those $M\in G$) with block form
    \begin{equation}
      M=\begin{pmatrix}
      \begin{pmatrix}
        I_{n-1} & 0\\
        x& 1
      \end{pmatrix} &
      \begin{pmatrix}
        0 & a_2\\
        a_3 & a_4
      \end{pmatrix}\\
      \begin{pmatrix}
        0&0\\
        0&0
      \end{pmatrix} &
      \begin{pmatrix}
        I_{n-1} & y\\
        0 & 1
      \end{pmatrix}
    \end{pmatrix}.
    \end{equation}

    The element $g_j$, in particular, has this form, so is in the kernel as claimed. Moreover, from \cref{eq:symp-rel} we may conclude that $y=-x^T$, $a_2=a_3^T$, and $a_4-a_4^T= v a_3^T- a_3 v^T$.  A simple induction then shows that
    \begin{equation}
      M^s=\begin{pmatrix}
      \begin{pmatrix}
        I_{n-k} & 0\\
        sx& I_k
      \end{pmatrix} &
      \begin{pmatrix}
        0 & sa_2\\
        sa_2^T & s a_4
      \end{pmatrix}\\
      \begin{pmatrix}
        0&0\\
        0&0
      \end{pmatrix} &
      \begin{pmatrix}
        I_{n-k} & sy\\
        0 & I_k
      \end{pmatrix}
    \end{pmatrix},
    \end{equation}
    for all $s\in\BN$.  Therefore $M$ has order dividing $p$ since $\BF_q$ has characteristic $p$.
\end{proof}
We can now prove our main result on the (projective) symplectic groups.
\begin{thm}\label{thm:symp-main}
  Let notation be as at the start of the section and as in \cref{prop:quot}.  Suppose that $p\equiv 1\bmod 4$ and that $q$ is an odd power of $p$.
  Then $G$ is non-$FSZ_{p^j}$ at $g_j$, and $H$ is non-$FSZ_{p^j}$ at $[g_j]$.
\end{thm}
\begin{proof}
  By \cref{prop:quot} we have a surjective group homomorphism $\pi\colon C_G(g_j)\to \Sp_{p^j-1}(q)\times \{\pm 1\}$. Let $\pi'\colon C_G(g_j)\to \Sp_{p^j-1}(q)$ be $\pi$ followed by the coordinate projection.  Consider any solution to $a^{p^j}=g_j$ in $G$ (which forces $a\in C_G(g_j)$), at least one of which exists by \cref{thm:all-res}.  Since $\ker(\pi)$ has exponent $p$ and contains $g_j$, we conclude that $a$ has order $p^{j+1}$ and $\pi(a)$ has order $p^j$. Since $p$ is odd by assumption, we may therefore conclude that $\pi'(a)$ also has order $p^j$.  Moreover, the Sylow $p$-subgroup of $\Sp_{p^j-1}(q)$ has exponent $p^j$, so $\pi'(a)$ has maximum possible order for any solution to $a^{p^j}=g_j$.  By dimension considerations we in fact see that the Jordan form of $\pi'(a)$ is the same for all solutions $a^{p^j}=g_j$, and consists of a single block.

  Let $W$ be the irreducible Weyl representation of $\Sp_{p^j-1}(q)$ of dimension $(p^j+1)/2$.  The element $-I_{p^j-1}$ acts trivially on this module.  Let $\chi_W$ be the character of $W$.  If $a^{p^j}=g_j$, then by \citep[Lemma 2.10 and Corollary 2.13.1]{SHINODA1980251} we conclude that $\chi_W(\pi'(a))=\gamma_1+\gamma_2\sqrt{q}$ such that

  \begin{gather*}
    \gamma_1\in\BQ^\times,\\
    \gamma_2\in i\BQ^\times \mbox{ when } p\equiv 3\bmod 4,\\
    \gamma_2\in\BQ^\times \mbox{ when } p\equiv 1\bmod 4,
  \end{gather*}
  and where $\gamma_1,\gamma_2$ do not depend on the choice of solution $a^{p^j}=g_j$.  We conclude that
  \begin{equation*}
    \alpha\ip{\chi_W\circ\pi'}{\psi_{p^j,g_j}} = \gamma_1+\gamma_2\sqrt{q}
  \end{equation*}
  for some $\alpha\in \BQ^\times$.  Subsequently,
  \begin{equation*}
    \beta_{p^j}(\chi_W\circ\pi',g_j) \sim_\BQ \lVert \gamma_1+\gamma_2\sqrt{q}\rVert^2,
  \end{equation*}
  and this is irrational if and only if $p\equiv 1\bmod 4$ and $q$ is an odd power of $p$.

  As $\chi_W\circ\pi'$ is necessarily an irreducible character of $C_G(g_j)$, this proves the desired claims for $G$.

  Since $C_H([g_j])=C_G(g_j)/Z(G)$ and $Z(G)$ is in the kernel of $\chi_W\circ\pi'$ by the definition of $W$, we immediately obtained the desired claims for $H$, as well.
\end{proof}

As was the case with the Sylow $p$-subgroups, the requirement that the dimension satisfies $2n=p^j+1$ seems unlikely to be necessary in general.  In the preceding proof this assumption was necessary to assure the value of $\chi_W\circ\pi'$ was independent of the choice of solution to $a^{p^j}=g_j$. In the more general case the sign of the $\gamma_2$ term will depend on the choice of $a$, and it is again conceivably possible (in the sense the author was again unable to rule it out) that this could result in things perfectly cancelling out to yield a rational value of $\beta_{p^j}(\chi_W\circ\pi',g_j)$.  We pose the following variation of \cref{q:sylow}.

\begin{question}
  Let $G=\Sp_{2n}(q)$.  Set $r=\lceil\log_p(2n)\rceil$.  Which, if any, of the following are equivalent?
  \begin{enumerate}
    \item $G$ is non-$FSZ_{p^j}$ for some $1\leq j<r$;
    \item $G$ is non-$FSZ_{p^j}$ for all $1\leq j<r$;
    \item $G$ has a non-$FSZ$ Sylow $p$-subgroup.
    \item $r>1$, $q$ is an odd power of $p$, and $p\equiv 1\bmod 4$.
  \end{enumerate}
\end{question}

\bibliographystyle{plainnat}
\bibliography{../references,../refs-symplectic}
\end{document}